\documentclass[10pt]{article}
\usepackage{amsmath,graphicx}
\usepackage{amssymb}
\usepackage[center]{subfigure}
\usepackage{amsthm}
\usepackage{color}
\usepackage{tikz}
\usepackage[english]{babel}
\usepackage[all]{xy}
\usepackage{cite}

 \newtheorem{proposition}{Proposition}
\newtheorem{corollary}{Corollary}
\title{The epimorphisms of the category Haus are exactly the image-dense morphisms}
\author{J\'er\^ome~Lapuyade-Lahorgue}
\date{}
\begin{document}
%
\maketitle
\section{Introduction}
In this document, we define the category \textbf{Haus} as the category whose objects are Hausdorff topological spaces and morphisms are the continous maps. The category \textbf{Top} is the category whose objects are the topological spaces and the morphisms are the continous maps. The definition of epimorphisms and the details about category theory can be found in \cite{adamek}.
\section{Image-dense morphisms of Haus are epimorphisms}
\begin{proposition}
Let $A$ and $B$ be two Hausdorff topological spaces and let $A\xrightarrow{f} B$ be a continuous map. If the closure $\overline{f(A)}$ of $f(A)$ is such $B=\overline{f(A)}$, then $f$ is an epimorphism.
\end{proposition}
\begin{proof}
Let $B\xrightarrow{g} C$ and $B\xrightarrow{h} C$ two continuous maps between two Hausdorff topological spaces such $g\neq h$. Then there exists $b\in B$ such $g(b)\neq h(b)$. As $C$ is Hausdorff, then there exists two disjoint open subsets $\mathcal{O}_{1}$ and $\mathcal{O}_{2}$ of $C$ such $g(b)\in \mathcal{O}_{1}$ and $h(b)\in \mathcal{O}_{2}$. We deduce that $b\in g^{-1}(\mathcal{O}_{1})\cap h^{-1}(\mathcal{O}_{2})$. As $B=\overline{f(A)}$, there exists $a\in A$ such $f(a)\in g^{-1}(\mathcal{O}_{1})\cap h^{-1}(\mathcal{O}_{2})$. And as $\mathcal{O}_{1}$ and $\mathcal{O}_{2}$ are disjoint, $g\circ f(a)\neq h\circ f(a)$. Consequently, $f$ is an epimorphism.
\end{proof}
\section{The category \textbf{Haus} is a reflective and full sub-category of \textbf{Top}}
The fullness of \textbf{Haus} in \textbf{Top} is trivial. We have to prove the reflectivity.
\begin{proposition}
Let $C$ be a topological space and let $\sim$ be the equivalence relation defined by:
\begin{equation}
x\sim y\Leftrightarrow \textrm{ for any continous map } C\xrightarrow{g}D\textrm{ where } D\textrm{ Hausdorff, }g(x)=g(y),\nonumber
\end{equation}
then the projection:
\begin{equation}
C\xrightarrow{r}C/\sim,\nonumber
\end{equation}
where $C/\sim$ is provided with the quotient-topology, is a reflection from the \textbf{Top}-object $C$ to the \textbf{Haus}-object $C/\sim$.
\end{proposition}
\begin{proof}
Let $C\xrightarrow{f}D$ a continuous map to a Hausdorff space $D$. If $r(x)=r(y)$, then, by definition of the equivalence relation $\sim$, $f(x)=f(y)$. Consequently, any $C\xrightarrow{f}D$ is compatible with $\sim$ and by consequence, there exists a unique function (ie. \textbf{Set}-morphism) $\overline{f}$ such the following diagram:
$$\xymatrix{
C\ar[r]^{r}\ar[d]^{f}&C/\sim\ar[ld]^{\overline{f}}\\
D
}$$
commutes. We have to show that $\overline{f}$ is a continuous map and that $C/\sim$ is a Hausdorff space.\\
$\overline{f}$ is a continuous map:\\
We recall that $\mathcal{O}$ is an open subset of $C/\sim$ is and only if $r^{-1}(\mathcal{O})$ is an open subset of $C$ (see \cite{godbillon}, Theorem 4.1 of the chapitre I). Let $U$ an open subset of $D$, then $r^{-1}(\overline{f}^{-1}(U))=f^{-1}(U)$ is an open set of $C$. We deduce that $\overline{f}^{-1}(U)$ is an open set of $C/\sim$ and by consequence, $\overline{f}$ is continous.\\
$C/\sim$ is a Hausdorff space:\\
Let $r(x)\neq r(y)$, then there exists $C\xrightarrow{f}D$ continous map to $D$ an Hausdorff space such $f(x)\neq f(y)$. $D$ is Hausdorff, then there exists two disjoint open-sets $\mathcal{O}_{x}$ and  $\mathcal{O}_{y}$ of $D$ such $f(x)\in \mathcal{O}_{x}$ and $f(y)\in \mathcal{O}_{y}$. It implies that $r(x)\in \overline{f}^{-1}(\mathcal{O}_{x})$ and $r(y)\in \overline{f}^{-1}(\mathcal{O}_{y})$. The subsets $\overline{f}^{-1}(\mathcal{O}_{x})$ and $\overline{f}^{-1}(\mathcal{O}_{y})$ are clearly disjoint open subsets of $C/\sim$.
\end{proof}
We will denote $H(C)$ the set  $C/\sim$ and it is called the ``Hausdorff quotient'' of $C$.
\begin{corollary}
A topological space $C$ is Hausdorff if and only if $C$ and $H(C)$ are homeomorphic.
\end{corollary}
\begin{proof}
As \textbf{Haus} is a full subcategory of \textbf{Top}, one can use the proposition I.4.20 of \cite{adamek}.
\end{proof}
\section{The epimorphisms of \textbf{Haus} are image-dense morphims}
\begin{proposition}
Let $A\xrightarrow{f}B$ an epimorphism of \textbf{Haus}, then $B=\overline{f(A)}$.
\end{proposition}
\begin{proof}
Let $B\xrightarrow{q}B/\overline{f(A)}$ be the quotient morphism such $q(x)=q(y)$ if and only if ($x=y$ and $x,y\notin \overline{f(A)}$) or ($x,y\in \overline{f(A)}$) and $B/\overline{f(A)}\xrightarrow{r}H\left(B/\overline{f(A)}\right)$ be the Hausdorff quotient of $B/\overline{f(A)}$.\\
We have $r\circ q\circ f$ is constant and as $f$ is an epimorphism of \textbf{Haus}, then the \textbf{Haus}-morphism $r\circ q$ is constant. So, either $q$ is constant or $r$ is constant. If $q$ is constant, then $B=\overline{f(A)}$.\\
Suppose that $B\neq\overline{f(A)}$ and $r$ is constant. Let $x\in B\backslash\overline{f(A)}$. As $\overline{f(A)}$ is a closed subset of $B$, then there exists an open subset $\mathcal{O}$ of $B$ such $x\in\mathcal{O}\subset B\backslash\overline{f(A)}$. By definition of $q$, for $x_1$ and $x_2$ in $\mathcal{O}$, $q(x_1)=q(x_2)$ if and only if $x_1=x_2$. It implies that $q^{-1}(q(\mathcal{O}))=\mathcal{O}$ and $q(\mathcal{O})$ is a Hausdorff open subset of $B/\overline{f(A)}$. We deduce that $r\circ q(\mathcal{O})$ has at least one point which is different from $r\circ q(\overline{f(A)})$. Consequently, $r$ takes at least two different values; this fact is contradictory with $r$ constant.
\end{proof}

\end{document}